\documentclass[11pt]{article}
\usepackage{amsfonts}
\usepackage{amsmath,amsthm,latexsym,amssymb}
\evensidemargin -0.04cm \numberwithin{equation}{section}

\newtheorem{theorem}{Theorem}
\newtheorem{lemma}{Lemma}

\newtheorem{proposition}{Proposition}

\numberwithin{theorem}{section} \numberwithin{corollary}{section}
\numberwithin{lemma}{section} \numberwithin{definition}{section}
\numberwithin{proposition}{section} \numberwithin{remark}{section}

\newcommand{\R}{\mathbb R}

\newcommand{\medint}{-\kern  -,375cm\int}

\newcommand{{\etaa}}{\eta_a}
\newcommand{\etaaT}{{\eta_a}_T}
\newcommand{\dvol}{dV}
\newcommand{\Sn}{\mathbb S^n}
\newcommand{\cn}{\omega_{n-1}}
\begin{document}
\title{Some sharp Hardy inequalities on spherically symmetric domains}
\author{Francesco Chiacchio\thanks{
Phone: +39-081\,675678; Fax: +39-081\,7662106; 
E-mail: \texttt{francesco.chiacchio@unina.it}.
Partially supported by GNAMPA-INdAM, Progetto \textquotedblleft
Propriet\`{a} analitico geometriche di soluzioni di equazioni ellittiche e
paraboliche \textquotedblright.} \ and Tonia Ricciardi\thanks{
Corresponding author. Phone: +39-081\,675663; Fax: +39-081\,675665;
E-mail: \texttt{tonia.ricciardi@unina.it};
Homepage: http://wpage.unina.it/tonricci.
Partially supported by Regione Campania L.R.~5/2006, by GNAMPA-INdAM, and by
a PRIN-COFIN project of the Italian Ministry of University and Research.} \\
{\small {Dipartimento di Matematica e Applicazioni ``R. Caccioppoli"}}\\
{\small {Universit\`a di Napoli Federico II, Via Cintia, 80126 Napoli, Italy.
}}
}
\date{}
\maketitle
\abstract{We prove some sharp Hardy inequalities for domains with a spherical symmetry.
In particular, we prove an inequality for domains of the unit $n$-dimensional sphere
with a point singularity, and an inequality
for functions defined on the half-space $\R_+^{n+1}$} vanishing on the hyperplane
$\{x_{n+1}=0\}$,
with singularity along the $x_{n+1}$-axis.
The proofs rely on a one-dimensional Hardy inequality
involving a weight function related to the volume element on the sphere,
as well as on symmetrization arguments.
The one-dimensional inequality is derived in a general form.
\begin{description}
\item {\textsc{Key Words:}} 
sharp weighted Hardy inequalities, symmetrization.
\item {\textsc{MSC 2000 Subject Classification:}} 46E35 (26D10, 35J25).
\end{description}
\section{Introduction and main results}
\label{sec:intro}
Sharp Hardy inequalities have attracted a considerable attention in recent years,
particularly in view of their applications to differential equations motivated by
physics and geometry.
Let $1<p<n$.
The classical Hardy inequality states that
\begin{equation}
\label{classicalHardy}
\int_{\mathbb{R}^{n}}\left\vert Du\right\vert_n^{p}\geq \left( \frac{n-p}{p}
\right) ^{p}\int_{\mathbb{R}^{n}}\frac{|u|^{p}}{|x|_n^p}
\end{equation}
for all smooth functions $u$ compactly supported on $\R^n$,
where we set $|x|_n^2=x_1^2+\ldots+x_n^2$ for all $x\in\R^n$.
A considerable effort has been devoted to extending this inequality to
manifolds, to special weight functions as well as to
domains exhibiting particular symmetries.
See, e.g., \cite{KP} for an extensive review.
Our aim in this note is to derive some sharp Hardy type inequalities
specifically tailored for manifolds with a spherical symmetry.
It should be mentioned that several recent inequalities
have been concerned with the special case of
the sphere, see, e.g., \cite{BPS,BP}. Indeed,
certain specific phenomena which do not occur on Euclidean space
actually do occur on spheres.
For example, in \cite{BPS} it was shown that the Sobolev inequality
admits minimizers on sufficiently large spherical caps.
\par
In order to state our main results we introduce some notation.
Let
\[
\mathbb S^n= \{x=(x_1,...,x_{n+1})\in\mathbb{R}^{n+1} : |x|=1 \}
\]
denote  the unit $n$-sphere, where we set
$|x|^2=|x|_{n+1}^2=x_1^2+\ldots+x_{n+1}^2$ for all $x\in\R^{n+1}$.
For $1<p<n$ and for $a>0$ we define the following weight function 
$\widetilde{\eta}_{a}:(0,a)\to\R$,
which is related to the volume element on $\mathbb S^n$,
as follows:
\begin{equation}
\label{tildeetaa}
\widetilde{\eta}_{a}(t)=\frac{(\sin t)^{-\frac{n-1}{p-1}}}
{\int_t^{a}(\sin s)^{-\frac{n-1}{p-1}}\,ds}.
\end{equation}
Note that 
$\lim_{t\to0^+}\widetilde{\eta}_{a}(t)
=\lim_{t\to a^-}\widetilde{\eta}_{a}(t)=+\infty$,
see \eqref{etaestimate} below.
In Lemma~\ref{phi_log_conc} 
we will show that  there exists $\widetilde T\in(0,a)$
such that $\widetilde{\eta}_{a}$ decreases in $(0,\widetilde T)$
and increases in $(\widetilde T,a)$.
Therefore, the following truncated function:
\begin{equation}
\label{tildeetaaTT}
\widetilde{\eta}_{aT}(t)=
\begin{cases}
\frac{(\sin t)^{-\frac{n-1}{p-1}}}
{\int_t^{a}(\sin s)^{-\frac{n-1}{p-1}}\,ds},
&\mathrm{if\ }t\in(0,\widetilde T)\\
\widetilde{\eta}_{a}(\widetilde T),
&\mathrm{if\ }x\in[\widetilde T,a)
\end{cases}
\end{equation}
is decreasing in $(0,a)$.
We denote by
$\Theta=(\theta_1,\ldots,\theta_{n-1},\theta_n)$ the angular variables on
$\mathbb S^n$ and to simplify notation we set $\theta=\theta_n$.
The angle $\theta\in [0,\pi]$, satisfying
$x_{n+1}=|x|\cos\theta$,
will be the only relevant angular variable
to our purposes.
We denote by $g$ the standard metric on $\mathbb S^n$ and by
$\dvol$ the volume element on $\mathbb S^n$.
For $\alpha\in(0,\pi]$ we denote by $\mathcal{B}(\alpha)$ the geodesic ball (spherical cap)
on $\mathbb{S}^n$ with radius $\alpha$
centered at the ``north pole'' $N=(0,\ldots,0,1)\in\R^{n+1}$. Namely,
we define
\begin{equation*}
\mathcal{B}(\alpha)=\left\{x\in \mathbb{S}^{n}:\ 0\leq
\theta<\alpha \right\}.
\end{equation*}
Let $N\in\Omega\subset\Sn$  be an open set such that $|\Omega|<|\mathbb S^n|$
and let $a^\star\in(0,\pi)$ be such that
$|\mathcal{B}(a^\star)|=|\Omega|$.
Here, for every measurable set $E\subset\Sn$,
$|E|$  denotes the volume of $E$ with respect to the standard Lebesgue measure
induced by $g$ on $\Sn$.
In turn, we define the following weight function
$\rho_{a^\star}:\mathbb S^n\setminus\{N\}\to\R$
\begin{equation}
\rho_{a^\star}(x)=
\begin{cases}
\frac{p-1}{n-p}\widetilde{\eta}_{a^\star T}(\theta)
&\mathrm{if\ }x\in \mathcal B(a^\star)\setminus\{N\}\\
\frac{p-1}{n-p}\widetilde{\eta}_{a^\star T}(\widetilde T)
&\mathrm{if\ }x\in\mathbb S^n\setminus\mathcal B(a^\star)
\end{cases},
\label{rho*}
\end{equation}
where $\widetilde{\eta}_{a^\star T}$ is the weight function
defined in \eqref{tildeetaaTT} with $a=a^\star$.
With this notation, our first result is the following.
\begin{theorem}
\label{thm:sphere}
Let $n\geq2$ and $1<p<n$. Let
$\Omega \subset \mathbb{S}^{n}$ be an open
set such that $N\in \Omega $ and
$|\Omega|<|\mathbb{S}^{n}|$.
Let $a^\star$ be such that $|\Omega|=|\mathcal B(a^\star)|$.
Then, for every $u\in W_{0}^{1,p}(\Omega)$, we have
\begin{equation}
\label{Hsphere}
\int_{\Omega }\left\vert \nabla u\right\vert ^{p}\,\dvol\geq \left(
\frac{n-p}{p}\right) ^{p}\int_{\Omega }\left\vert u\right\vert ^{p}\rho_{a^\star}^{p}\,\dvol.
\end{equation}
The constant $[(n-p)/p]^{p}$ is sharp.
\end{theorem}
Note that, since $\theta = d_g(x,N)$, we have 
\begin{equation}
\label{rhoasympt}
\lim_{x \to N} d_g(x,N)  \rho_{a^\star}(x)=1,
\end{equation}
see \eqref{rho_asym} below,
so that $\rho_{a^\star}$ is a natural extension of the classical singularity
$|x|_n^{-p}$ appearing in \eqref{classicalHardy}.
\par
Theorem~\ref{thm:sphere}, together with a Steiner symmetrization with respect
to the angular variables,
yields an inequality for functions defined on the half-space
$\R^{n+1}_{+}=\{x\in\R^{n+1}:\ x_{n+1}>0\}$
with singularity along the $x_{n+1}$-axis.
More precisely, for every $x\in\R^{n+1}$ let $x'=(x_1,\ldots,x_n,0)$.
We take $a^\star=\pi/2$ in \eqref{tildeetaa}--\eqref{rho*} and we define
the singularity $\zeta:\R_+^{n+1}\setminus\{x'=0\}\to\R$ as follows:
\[
\zeta(x)=\rho_{\frac{\pi}{2}}\left(\frac{x}{|x|}\right).
\]
Note that $\zeta$ is singular on the $x_{n+1}$-axis.
We have:
\begin{theorem}
\label{thm:Fubini} For every $u\in W_{0}^{1,p}(\mathbb{R}_{+}^{n+1})$, the
following inequality holds:
\begin{equation}
\int_{\mathbb{R}_{+}^{n+1}}\left\vert D_{\Theta }u\right\vert ^{p}\,dx
\geq
\left( \frac{n-p}{p}\right) ^{p}\int_{\mathbb{R}_{+}^{n+1}}\left\vert
u\right\vert ^{p}\frac{\zeta^{p}(x)}{|x|^{p}}\,dx.  
\label{Est_on_Rn_+}
\end{equation}
Here, $D_{\Theta}u(x)$ denotes the projection of the gradient $Du(x)$ on the
sphere $\partial B(0,|x|)$. The constant $\left( \frac{n-p}{p}\right)^{p}$
is sharp.
\end{theorem}
We observe that for special values of $p$ and $n$ the singularities appearing
in \eqref{Hsphere} and \eqref{Est_on_Rn_+} take particularly simple and explicit forms. 
More precisely, let $p=(n+1)/2$ and
suppose that $\Omega\subset\mathbb S^n$ is such that $|\Omega|=|\mathbb S^n|/2$.
Then, $a^{\star }=\pi /2$, $(p-1)/(n-p)=1$, $(n-1)/(p-1)=2$
and therefore
$\widetilde{\eta}_{\pi/2}(t)
=\sin^{-2}t(\int_t^{\pi/2}\sin^{-2}\sigma\,d\sigma)^{-1}=(\sin t\cos t)^{-1}$.
Consequently, $\widetilde{T}=\pi/4$ and
\[
\rho_{\pi/2}(x)=
\begin{cases}
(\sin\theta\cos\theta )^{-1}&\mathrm{if\ }x\in\mathcal{B}(\pi/4)\setminus\{N\}\\
2&\mathrm{if\ }x\in\mathbb S^n\setminus\mathcal B(\pi/4)
\end{cases}.
\]
Note also that $(\sin\theta\cos\theta )^{-1}>\theta^{-1}$ for any
$\theta\in\left(0,\pi/4\right)$ and $2>\theta^{-1}$ for any
$\theta\in\left(\pi/4,\pi\right).$ Therefore,
inequality~\eqref{Hsphere} implies that
\begin{equation*}
\int_{\Omega}\left\vert\nabla u\right\vert^{p}\,dV
\geq\left(\frac{n-p}{p}\right)^{p}\int_{\Omega}\left\vert u\right\vert^{p}
\left[\frac{1}{d_g(x,N)^{p}}+h\right]\,dV,
\end{equation*}
where $h$ is a positive quantity, thus showing that 
\eqref{classicalHardy} is improved on the sphere in this case. 
When $p=(n+1)/2$ the same arguments also yield a simple form for 
\eqref{Est_on_Rn_+}. Indeed, in this case \eqref{Est_on_Rn_+} may be written in the form
\begin{equation*}
\int_{\mathbb{R}_{+}^{n+1}}\left\vert D_{\Theta }u\right\vert^{p}\,dx
\geq\left(\frac{p-1}{p}\right)^{p}\left(\int_{\mathbb{R}_{+}^{n+1}\cap
\left\{0\leq\theta<\pi/4\right\}}\frac{\left\vert u\right\vert^{p}}
{\left\vert x^{\prime }\right\vert ^{p}(\cos \theta )^{p}}\,dx
+2^{p}\int_{\mathbb{R}_{+}^{n+1}\cap\left\{\pi/4\leq\theta
<\pi/2\right\}}\frac{\left\vert u\right\vert^{p}}{\left\vert x\right\vert^{p}}
\,dx\right).
\end{equation*}
This special case was also shown to be of interest in \cite{BPS}.
\par
An outline of the proofs may be as follows.
Our starting point is the following one-dimensional Hardy inequality:
\begin{equation}
\label{IDsinhardy}
\int_0^a|u'|^p\sin^{n-1}(t)\,dt\ge\left(\frac{p-1}{p}\right)^p
\int_0^a|u|^p\widetilde\eta_{aT}^p\sin^{n-1}(t)\,dt,
\end{equation}
for all $u$ such that $u(a)=0$,
where $\widetilde\eta_{aT}$ is the function defined in \eqref{tildeetaaTT}.
In fact, in Section~\ref{sec:1D} we shall prove
some sharp weighted one-dimensional Hardy inequalities involving
a \textit{general} weight $\phi$ which reduces to \eqref{IDsinhardy} when
$\phi(t)=\widetilde\phi(t)=\sin^{n-1}(t)$, see Proposition~\ref{prop:hardy} 
and Proposition~\ref{One_dim_Hardy_Tr} below.
To this end, we extend a method described in \cite{HLP}, see also \cite{Tar},
for the special case $\phi(t)=1$.
In fact, one of our efforts is to determine very general conditions on
$\phi$ such that this method is applicable.
This technique was also employed in \cite{BCT} in the special case
$\phi(t)=(2\pi)^{-1/2}\exp\{-t^2/2\}$
in the context of symmetrization with respect to Gaussian measure.
On the other hand, our sharpness considerations as in Proposition~\ref{One_dim_Hardy_Tr}
are new even in these special cases.
In Section~\ref{sec:wholespace}
we employ spherical symmetrization in order to reduce Theorem~\ref{thm:sphere}
to \eqref{IDsinhardy}.
In turn, Theorem~\ref{thm:sphere} together with 
a Steiner symmetrization 
with respect to the angular variables 
concludes the proof of Theorem~\ref{thm:Fubini}.
\section{Some Hardy inequalities on intervals}
\label{sec:1D}
Our aim in this section is to prove some weighted
one-dimensional Hardy inequalities as stated in Proposition~\ref{prop:hardy} 
and Proposition~\ref{One_dim_Hardy_Tr} below.
To this end, as already mentioned in Section~\ref{sec:intro}, we exploit a technique
from \cite{HLP}, Theorem~253 p.~175, see
also \cite{BCT,Tar}. Let $a>0$, $p>1$ and let
$\phi\in C^{1}\left( 0,a\right] \cap C^{0} ( \left[ 0,a \right] )$
 be such that
\begin{equation}
\label{phibasic}
\phi(0)=0, \qquad\qquad\phi (t)>0\text{ in }\left( 0,a\right], \qquad\qquad
c_1t^{p-1+\delta}\le\phi(t)\le c_2t^{p-1+\delta},
\end{equation}
for some $c_1,c_2,\delta >0$. We denote
\begin{equation*}
W^{1,p}\left( 0,a;\phi \right) =\left\{ u:\left[ 0,a\right] \rightarrow
\mathbb{R};\ u\in L_{\text{loc}}^{1}\left[ 0,a\right]\ \mathrm{\ and\ }
\int_{0}^{a}\left\vert u^{\prime }\right\vert ^{p}\phi\,dt<+\infty \right\},
\end{equation*}
where $u^{\prime}$ denotes the distributional derivative of $u$.
We consider the following subspace of $W^{1,p}\left( 0,a;\phi \right)$
\begin{equation*}
\mathcal{E} = \left\{
u \in W^{1,p}\left( 0,a;\phi \right) : u(a)=0
\right\},
\end{equation*}
endowed with the norm
$\left\Vert u\right\Vert =\left( \int_{0}^{a}\left\vert u^{\prime
}\right\vert ^{p}\phi \right) ^{\frac{1}{p}}$.
We note that if $u \in \mathcal{E}$, then $u$ is absolutely continuous
in $\left[ \epsilon,a\right]$, for all
$\epsilon\in\left(0,a\right)$.
On the other hand, $u$ is in general unbounded near the origin.
Nevertheless, $u$ may be approximated in $\mathcal{E}$ by functions
which vanish in $0$. More precisely we have the following.
\begin{lemma}
\label{lem:approx} $C_{0}^{1}[0,a]$ is dense in $\mathcal{E}$.
\end{lemma}
\begin{proof}
Let $u \in \mathcal{E}.$ By standard properties of Sobolev spaces we may assume that
$u \in C^{1}[\epsilon,a]$ for all $\epsilon \in (0,a)$.
We first show that $C_{0}^{1}[0,a]$ is dense in
$\mathcal{E} \cap L^\infty(0,a)$.
Let $u\in \mathcal{E} \cap L^{\infty }(0,a)$.
We consider the sequence:
\begin{equation*}
u_{k}(t)=\left\{
\begin{array}{ccc}
u\left( k^{-1}\right)kt & \text{if} & t\in \left[ 0,k^{-1}\right] \\
&  &  \\
u(t) & \text{if} & t\in \left] k^{-1},a\right] ,
\end{array}
\right.
\end{equation*}
where $k\in \mathbb{N}$.
By the elementary inequality
$\left\vert \alpha +\beta \right\vert ^{p}\leq 2^{p-1}\left( \left\vert
\alpha \right\vert ^{p}+\left\vert \beta \right\vert ^{p}\right) $,
 for all $\alpha,$  $\beta$ $\in \mathbb{R}$, we have
\begin{equation}
\int_{0}^{a}\left\vert u^{\prime }-u_{k}^{\prime }\right\vert
^{p}\phi\,dt
=
\int_{0}^{k^{-1}}\left\vert u^{\prime }-u\left(
k^{-1}\right) k\right\vert ^{p}\phi\,dt
\leq
 2^{p-1}\left\{
\int_{0}^{k^{-1}}\left\vert u^{\prime }\right\vert
^{p}\phi\,dt+
\left\vert \left\vert u\right\vert \right\vert _{\infty }^{p}
k^{p}\int_{0}^{k^{-1}}\phi\,dt\right\} .  \label{d_0}
\end{equation}
Since $u^{\prime }\in L^{p}\left( 0,a;\phi \right)$, the absolute
continuity of the Lebesgue integral implies that
\begin{equation}
\int_{0}^{k^{-1}}\left\vert u^{\prime }\right\vert
^{p}\phi\,dt=o(1),\quad \textrm{as} \quad k\rightarrow +\infty.  \label{d_1}
\end{equation}
In view of \eqref{phibasic}, we have
\begin{equation*}
\max_{t\in \left[ 0,k^{-1}\right] }\phi \leq c_2 \max_{t\in \left[
0,k^{-1} \right] }t^{p-1+\delta }=c_2 k^{-p+1-\delta }.
\end{equation*}
Consequently,
\begin{equation}
k^{p}\int_{0}^{k^{-1}}\phi \,dt
\leq
k^{p-1}\max_{t\in \left[
0,k^{-1}\right] }\phi
\leq
 c_2 k^{-\delta }=o(1), \quad \textrm{as} \quad k\rightarrow +\infty.  \label{d_3}
\end{equation}
In view of (\ref{d_0})--(\ref{d_1})--(\ref{d_3}) it follows that
\begin{equation*}
\int_{0}^{a}\left\vert u^{\prime }-u_{k}^{\prime }\right\vert^{p}\phi\,dt=o(1),
\quad \textrm{as} \quad k\rightarrow +\infty.
\end{equation*}
We conclude by a standard regularization argument.
Now suppose that $u\in \mathcal{E}$. The
following sequence of bounded functions
\begin{equation*}
\widehat{u}_{k}(t)=\left\{
\begin{array}{ccc}
u\left( k^{-1}\right) & \text{if} & t\in \left[ 0,k^{-1}\right] \\
&&  \\
u(t) & \text{if} & t\in \left] k^{-1},a\right] ,
\end{array}
\right.
\end{equation*}
satisfies
\begin{equation*}
\int_{0}^{a}\left\vert u^{\prime }-\widehat{u}_{k}^{\prime
}\right\vert ^{p}\phi\,dt=\int_{0}^{\frac{1}{k}}\left\vert
u^{\prime }\right\vert ^{p}\phi \,dt=o(1),  \quad \textrm{as} \quad k\rightarrow +\infty.
\end{equation*}
Hence, we are reduced to the case where $u$ is bounded, and the claim is established.
\end{proof}
Fix $a>0$. Let
\begin{equation}
\etaa (t)=\frac{\phi (t)^{-\frac{1}{p-1}}} {\int_{t}^{a}\phi(\sigma )^{-\frac{
1}{p-1}}\,d\sigma },\text{ }t\in \left( 0,a\right).  \label{eta}
\end{equation}
We note that $\etaa>0$ and furthermore the following holds.
\begin{lemma}
\label{lem:eta}
The function $\etaa$ defined by \eqref{eta} satisfies:
\begin{equation}
\label{etaestimate}
\left(\frac{c_1}{c_2}\right)^{\frac{1}{p-1}}\frac{\delta}{p-1}\, \frac{
a^{\delta/(p-1)}}{t\,(a^{\delta/(p-1)}-t^{\delta/(p-1)})} \le\etaa(t)
\le\left(\frac{c_2}{c_1}\right)^{\frac{1}{p-1}}\frac{\delta}{p-1}\, \frac{
a^{\delta/(p-1)}}{t\,(a^{\delta/(p-1)}-t^{\delta/(p-1)})}
\end{equation}
for all $t\in(0,a)$, where $c_1,c_2,\delta>0$ are the constants defined in
\eqref{phibasic}.
\end{lemma}
\begin{proof}
We have:
\begin{align*}
\int_t^a\left(\sigma^{p-1+\delta}\right)^{-\frac{1}{p-1}}\,d\sigma
=\int_t^a\sigma^{-1-\delta/(p-1)}\,d\sigma
=\frac{p-1}{\delta}\left(t^{-\delta/(p-1)}-a^{-\delta/(p-1)}\right).
\end{align*}
Consequently,
\begin{align*}
\frac{\left(t^{p-1+\delta}\right)^{-\frac{1}{p-1}}}
{\int_t^a\left(\sigma^{p-1+\delta}\right)^{-\frac{1}{p-1}}\,d\sigma}
=\frac{\delta}{p-1}\,
\frac{a^{\delta/(p-1)}}{t\,(a^{\delta/(p-1)}-t^{\delta/(p-1)})}.
\end{align*}
On the other hand, in view of the assumption~\eqref{phibasic} on
$\phi$, we have
\[
\left(\frac{c_1}{c_2}\right)^{\frac{1}{p-1}}\frac{\left(t^{p-1+\delta}\right)^{-\frac{1}{p-1}}}
{\int_t^a\left(\sigma^{p-1+\delta}\right)^{-\frac{1}{p-1}}\,d\sigma}
\le\etaa(t) \le
\left(\frac{c_2}{c_1}\right)^{\frac{1}{p-1}}\frac{\left(t^{p-1+\delta}\right)^{-\frac{1}{p-1}}}
{\int_t^a\left(\sigma^{p-1+\delta}\right)^{-\frac{1}{p-1}}\,d\sigma}
\]
and the asserted estimate follows.
\end{proof}
For later use, we also note that $\etaa$ satisfies a Riccati equation:
\begin{equation}
\etaa \frac{\phi ^{\prime }}{\phi }+(p-1)\etaa ^{\prime }=(p-1)\etaa ^{2}
\qquad\mathrm{in\ }(0,a).
\label{Eq_eta}
\end{equation}
The following Hardy inequality holds.
\begin{proposition}
\label{prop:hardy}
Let $a>0$, $p>1$ and
suppose that $\phi $ satisfies \eqref{phibasic}.
Let $\etaa$ be
correspondingly defined by \eqref{eta}. Then, for every $u\in \mathcal{E}$, the following inequality holds:
\begin{equation}
\int_{0}^{a}\left\vert u^{\prime }\right\vert ^{p}\phi\,dt\geq
\left( \frac{p-1}{p}\right) ^{p}\int_{0}^{a}\left\vert u\right\vert ^{p}\etaa ^{p}\phi\,dt.
\label{dis_p_1_dim}
\end{equation}
The constant $\left(\frac{p-1}{p}\right) ^{p}$ is sharp.
\end{proposition}
\begin{proof}
In view of Lemma~\ref{lem:approx} we may assume that $u\in
C_{0}^{1}\left( 0,a\right)$. We recall the elementary convexity
inequality $\left\vert\alpha\right\vert ^{p}\geq\left\vert
\beta\right\vert ^{p}+p\left\vert\beta\right\vert ^{p-2}\beta(\alpha-\beta)$ for all
$\alpha,\beta\in\mathbb{R}$.
Taking $\alpha=u^{\prime}$ and
$\beta=-\frac{p-1}{p}u\etaa$, we derive:
\begin{equation}
\left\vert u^{\prime }\right\vert ^{p}\geq \left\vert
\frac{p-1}{p}u\etaa \right\vert ^{p}-p\left\vert \frac{p-1}{p}u\etaa
\right\vert ^{p-2}\frac{p-1}{ p}u\etaa \left( u^{\prime
}+\frac{p-1}{p}u\etaa \right).  \label{conv_ineq_u'}
\end{equation}
Multiplying by $\phi$ and integrating over $\left[ 0,a\right]$, we
obtain:
\begin{equation*}
\int_{0}^{a}\left\vert u^{\prime }\right\vert ^{p}\phi\,dt\geq
\left(\frac{ p-1}{p}\right)^{p}(1-p)\int_{0}^{a}\left\vert
u\right\vert ^{p}\etaa ^{p}\phi\,dt-p\left( \frac{p-1}{p}\right)^{p-1}
\int_{0}^{a}\left\vert u\right\vert ^{p-2}uu^{\prime }\etaa^{p-1}\phi\,dt.
\end{equation*}
Integration by parts yields
\begin{eqnarray*}
\int_{0}^{a}\left\vert u^{\prime }\right\vert ^{p}\phi\,dt &\geq
&\left( \frac{p-1}{p}\right) ^{p}(1-p)\int_{0}^{a}\left\vert
u\right\vert ^{p}\etaa ^{p}\phi\,dt+\left( \frac{p-1}{p}\right)^{p-1}\int_{0}^{a}\left\vert
u\right\vert ^{p}\left( \etaa ^{p-1}\phi \right) ^{\prime }\,dt \\
&&-\left( \frac{p-1}{p}\right)^{p-1}
\left[
\left\vert u\right\vert^{p}\etaa ^{p-1}\phi \right]_{0}^{a}.
\end{eqnarray*}
Now we observe that by \eqref{etaestimate}
and the fact that $u\in C_{0}^{1}\left(\left[0,a\right]\right)$,
we have $u\etaa\in L^\infty(0,a)$.
Therefore, the
boundary terms vanish and we obtain
\begin{align*}
\int_{0}^{a}&\left\vert u^{\prime }\right\vert ^{p}\phi\,dt \\
\geq&(1-p)\left(\frac{p-1}{p}\right)^{p}
\int_{0}^{a}\left\vert
u\right\vert ^{p}\etaa ^{p}\phi\,dt+\left(\frac{p-1}{p}\right)^{p-1}
\int_{0}^{a}\left\vert u\right\vert ^{p}\left( \etaa^{p-1}\phi^{\prime }
+(p-1)\etaa^{p-2}\etaa ^{\prime }\phi \right)\,dt \\
=&(1-p)\left( \frac{p-1}{p} \right)^{p}
\int_{0}^{a}\left\vert u\right\vert^{p}\etaa^{p}\phi\,dt
+\left( \frac{p-1}{p}\right)^{p-1}\int_{0}^{a}\left\vert u\right\vert ^{p}
\etaa ^{p-2}\left(\etaa\frac{\phi^{\prime}}{\phi}+(p-1)\etaa ^{\prime}\right)\phi\,dt.
\end{align*}
In view of \eqref{Eq_eta}, we have
\[
\int_{0}^{a}\left\vert u\right\vert ^{p}\etaa ^{p-2}\left( \etaa
\frac{\phi ^{\prime }}{\phi } +(p-1)\etaa ^{\prime }\right)\phi\,dt
=(p-1)\int_{0}^{a}\left\vert u\right\vert ^{p}\etaa ^{p}\phi\,dt.
\]
It follows that
\begin{equation*}
\int_{0}^{a}\left\vert u^{\prime }\right\vert ^{p}\phi\,dt
\geq
\left[
\left(\frac{p-1}{p}\right)^{p}-\frac{(p-1)}{p^{p-1}}^{p}
\right]
\int_{0}^{a}\left\vert u\right\vert ^{p}\etaa^{p}\phi\,dt
+\frac{(p-1)^{p}}{p^{p-1}} \int_{0}^{a}\left\vert
u\right\vert ^{p}\etaa ^{p}\phi\,dt=\left(\frac{p-1}{p}\right)^{p}\int_{0}^{a}\left\vert u\right\vert^{p}\etaa ^{p}\phi\,dt.
\end{equation*}
Hence, \eqref{dis_p_1_dim} is satisfied.
\par
Now we verify sharpness. To this end, we consider the sequence of functions
$\{U_k\}_{k\in\mathbb N}\subset\mathcal E$ defined by
\begin{equation}
U_{k}(t)=\begin{cases} \left( \int_{\frac{1}{k}}^{a}
\phi^{-\frac{1}{p-1}}\,d\sigma \right)^{\frac{p-1}{p}}&\mathrm{if\ }t\in\left[0,\frac{1}{k}\right)\\
\left(\int_{t}^{a}\phi ^{-\frac{1}{p-1}}\,d\sigma \right)^{\frac{p-1}{p}}
&\mathrm{if\ }t\in \left[ \frac{1}{k},a\right]
\end{cases}.
\label{U_k}
\end{equation}
Then,
\begin{equation*}
U_{k}^{\prime }(t)=
\begin{cases}
0&\mathrm{if\ }t\in \left[ 0,\frac{1}{k}\right)\\
-\frac{p-1}{p}\left(\int_{t}^{a}\phi ^{-\frac{1}{p-1}}\,d\sigma
\right)^{-\frac{1}{p}}\phi^{-\frac{1}{p-1}}(t)& \mathrm{if\ }
t\in \left[ \frac{1}{k},a \right]
\end{cases}.
\end{equation*}
We claim that
\begin{equation*}
\lim_{k\rightarrow +\infty }\dfrac{\int_{0}^{a}\left\vert
U_{k}^{\prime }\right\vert ^{p}\phi\,dt}{\int_{0}^{a}U_{k}^{p}\etaa
^{p}\phi\,dt}=\left( \frac{p-1}{p}\right) ^{p}.
\end{equation*}
Indeed, we note that
\begin{equation*}
(\etaa^{p}\phi)(t)=\frac{\phi^{-\frac{1}{p-1}}(t)}{\left(
\int_{t}^{a}\phi ^{- \frac{1}{p-1}}d\sigma \right) ^{p}}
\end{equation*}
for all $t\in(0,a)$.
Therefore, we may write
\begin{equation*}
\int_{0}^{a}U_{k}^{p}\etaa^{p}\phi\,dt=A_{k}+B_{k},
\end{equation*}
where
\begin{equation}
A_{k}\equiv\int_{0}^{\frac{1}{k}}U_{k}^{p}\etaa^{p}\phi\,dt
=\left( \int_{\frac{1}{k}}^{a}\phi^{-\frac{1}{p-1}}\,d\sigma
\right) ^{p-1}\int_{0}^{\frac{1}{k}}
\frac{\phi^{-\frac{1}{p-1}}(t)}{\left( \int_{t}^{a}\phi^{-\frac{1}{p-1}}\,d\sigma\right)^{p}}\,dt
\label{A_k}
\end{equation}
and
\begin{equation}
B_{k} \equiv \int_{\frac{1}{k}}^{a}U_{k}^{p}\etaa^{p}\phi\,dt
=\int_{\frac{1}{k}}^{a}\left(\int_{t}^{a}\phi^{-\frac{1}{p-1}}\,d\sigma\right)^{p-1}
\frac{\phi^{-\frac{1}{p-1}}(t)}{\left(\int_{t}^{a}\phi^{-\frac{1}{p-1}}\,d\sigma \right)^{p}}\,dt
=\int_{\frac{1}{k}}^{a}\frac{\phi^{-\frac{1}{p-1}}(t)}{\int_{t}^{a}\phi ^{-\frac{1}{p-1}}d\sigma}\,dt.
\label{B_k}
\end{equation}
We claim that
\begin{equation}
\label{AkHopital}
\lim\limits_{k\rightarrow +\infty }A_{k}=\frac{1}{p-1}.
\end{equation}
Indeed, we first observe that in view of \eqref{phibasic}
we have
\begin{align}
\label{infty}
\int_{1/k}^a\frac{\phi^{-\frac{1}{p-1}}(t)}{\int_t^a\phi^{-\frac{1}{p-1}}\,d\sigma}\,dt
\ge\left(\frac{c_2}{c_1}\right)^{-\frac{1}{p-1}}\int_{1/k}^a\frac{t^{-1-\frac{\delta}{p-1}}}
{\int_t^a\sigma^{-1-\frac{\delta}{p-1}}\,d\sigma}\,dt\to+\infty
\end{align}
as $k\to\infty$.
Hence, by L'Hospital's rule,
\begin{equation*}
\lim\limits_{k\rightarrow +\infty }A_{k}=\lim_{k\rightarrow
+\infty }\frac{
\int_{0}^{\frac{1}{k}}\frac{\phi^{-\frac{1}{p-1}}}
{\left(\int_{t}^{a}\phi^{-\frac{1}{p-1}}\,d\sigma\right)^{p}}\,dt}
{\left(\int_{\frac{1}{k}}^{a}\phi ^{-\frac{1}{p-1}}\,dt\right)^{-(p-1)}}
=\frac{\frac{\phi (\frac{1}{k})^{-\frac{1}{p-1}}}{\left( \int_{\frac{1}{k}}^{a}
\phi^{-\frac{1}{p-1}}\,d\sigma \right)^{p}}}{(p-1)\left(\int_{\frac{1}{k}}^{a}
\phi^{-\frac{1}{p-1}}\,dt\right) ^{-p}\phi(\frac{1}{k})^{-\frac{1}{p-1}}}=\frac{1}{p-1},
\end{equation*}
and \eqref{AkHopital} follows.
We conclude that
\begin{equation*}
\int_{0}^{a}U_{k}^{p}\etaa^{p}\phi\,dt=\int_{\frac{1}{k}}^{a}
\frac{\phi^{-\frac{1}{p-1}}(t)}{\int_{t}^{a}\phi^{-\frac{1}{p-1}}\,d\sigma}\,dt+\frac{1}{p-1}+o(1).
\end{equation*}
On the other hand, we have
\begin{equation*}
\int_{0}^{a}\left\vert U_{k}^{\prime }\right\vert^{p}\phi\,dt
=\left(\frac{p-1}{p}\right)^{p}\int_{\frac{1}{k}}^{a}\frac{\phi^{-\frac{1}{p-1}}(t)}
{\int_{t}^{a}\phi^{-\frac{1}{p-1}}\,d\sigma}\,dt.
\end{equation*}
Hence, recalling \eqref{infty}, we obtain
\begin{equation*}
\lim_{k\rightarrow +\infty }\dfrac{\int_{0}^{a}\left\vert
U_{k}^{\prime }\right\vert ^{p}\phi\,dt}{\int_{0}^{a}U_{k}^{p}\etaa^{p}\phi\,dt}
=\left( \frac{p-1}{p}\right)^{p}\lim_{k\rightarrow+\infty }
\frac{\int_{\frac{1}{k} }^{a}\frac{\phi^{-\frac{1}{p-1}}}
{\int_{t}^{a}\phi ^{-\frac{1}{p-1}}\,d\sigma}\,dt}
{\int_{\frac{1}{k}}^{a}\frac{\phi^{-\frac{1}{p-1}}}{\int_{t}^{a}
\phi ^{-\frac{1}{p-1}}\,d\sigma}\,dt+\frac{1}{p-1}+o(1)}=\left(\frac{p-1}{p}\right)^{p}.
\end{equation*}
Hence, the sharpness is also established.
\end{proof}
Now we show that under an extra simple assumption for $\phi$, the corresponding
function $\etaa$ defined by \eqref{eta} has exactly one critical point,
corresponding to the absolute minimum of $\etaa$ in $(0,a)$.
\begin{lemma}
\label{phi_log_conc}
Suppose that $\phi:[0,a]\to\mathbb{R}$ satisfies
\eqref{phibasic}. Furthermore, suppose that $\phi$ is twice differentiable
in $(0,a)$ and that
\begin{equation}
\left(\log\phi\right)^{\prime\prime}(t)
=\left(\frac{\phi^{\prime }}{\phi}\right)^{\prime}(t)<0\qquad
\mathrm{\ for\ all\ }t\in(0,a).  \label{log_concave}
\end{equation}
Then, there exists a unique $T\in(0,a)$ such that $\etaa^{\prime}(t)<0$ in $(0,T)$
and $\etaa^{\prime}(t)>0$ in $(T,a)$.
\end{lemma}
\begin{proof}
Differentiating the Riccati equation~\eqref{Eq_eta} we obtain
\begin{equation}
\etaa^{\prime}\frac{\phi^{\prime}}{\phi}
+\etaa\left(\frac{\phi^{\prime}}{\phi}\right)^{\prime}
+(p-1)\etaa^{\prime\prime}=2(p-1)\etaa\etaa^{\prime}.
\label{eta_der_2}
\end{equation}
Suppose that $\etaa'(\widehat{t})=0$. Then,
(\ref{eta_der_2}) implies that
\begin{equation*}
\etaa^{\prime\prime}(\widehat{t})=-\frac{1}{p-1}\etaa
(\widehat{t})\left(\frac{\phi^{\prime }}{\phi}\right)^{\prime}(\widehat{t})>0.
\end{equation*}
It follows that any critical point for $\etaa$ is necessarily a
strict minimum point. In view of Lemma~\ref{lem:eta}, it
follows that $\etaa$ admits a unique minimum point and the
existence of $T$ is established.
\end{proof}
Let $\phi$ be twice differentiable and suppose that $\phi$
satisfies \eqref{phibasic} and \eqref{log_concave}.
Then, the following
function obtained by truncating $\etaa$ at the point $T$, is non-increasing:
\begin{equation}
\label{etaT}
\etaaT(t)=
\begin{cases}
\etaa(t)&\mathrm{for\ }t\in\left( 0,T\right] \\
\etaa(T)&\mathrm{for\ } t\in \left( T,a\right].
\end{cases}
\end{equation}
Since $\etaaT\le\etaa$ pointwise, it is clear that Proposition~\ref%
{prop:hardy} still holds with $\etaa$ replaced by $\etaaT$. On the other
hand, it is not a priori clear whether or not, with such a replacement, the constant
$[(p-1)/p]^p$ is still sharp. In the next proposition we show that this is
indeed the case.
\begin{proposition}
\label{One_dim_Hardy_Tr}
Suppose that $\phi$ is twice differentiable and satisfies
\eqref{phibasic} and \eqref{log_concave}. Let $\etaaT$ be defined by
\eqref{etaT}. Then,
\begin{equation*}
\int_{0}^{a}\left\vert u^{\prime }\right\vert^{p}\phi\,dt
\geq\left(\frac{p-1 }{p}\right)^{p}\int_{0}^{a}\left\vert u\right\vert^{p}
\etaaT^{p}\phi\, dt,\ \forall u\in\mathcal E.
\end{equation*}
Furthermore, the constant $\left(\frac{p-1}{p}\right)^{p}$ is sharp.
\end{proposition}
\begin{proof}
We need only check sharpness. For $k\in\mathbb N$, $1/k<T$,
we consider the sequence $\{V_k\}_{k\in\mathbb N}\subset\mathcal E$
defined by
\begin{equation}
V_{k}(t)=
\begin{cases}
U_k(t)
&\mathrm{if\ }t\in[0,T)\\
\left(\int_{T}^{a}\phi^{-\frac{1}{p-1}}\,d\sigma\right)^{\frac{p-1}{p}} \frac{2t-a-T}{T-a}
&\mathrm{if\ }t\in \left[T,(a+T)/2\right)
\\
0 & \text{if\ }t\in \left[(a+T)/2,a\right]
\end{cases},
\label{U_tilde_k}
\end{equation}
where $\{U_k\}_{k\in\mathbb N}$ is the sequence defined
in \eqref{U_k}.
Then,
\begin{equation*}
V_{k}^{\prime }(t)=
\begin{cases}
U_k'(t)
&\mathrm{if\ }t\in[0,T)\\
\left(\int_{T}^{a}\phi ^{-\frac{1}{p-1}}\,d\sigma \right)^{\frac{p-1}{p}}
\frac{2}{T-a}& \text{if\ }  t\in \left[ T,(a+T)/2\right) \\
0 & \text{if\ }  t\in \left[ (a+T)/2,a\right]
\end{cases}.
\end{equation*}
We claim that
\begin{equation*}
\lim_{k\rightarrow +\infty }\dfrac{\int_{0}^{a}\left\vert
V_{k}^{\prime }\right\vert ^{p}\phi\,dt}{\int_{0}^{a}\left\vert
V_{k}\right\vert ^{p}\etaaT^{p}\phi\,dt}=\left(
\frac{p-1}{p}\right) ^{p}.
\end{equation*}
Let $C_1,C_2>0$ be defined by
\begin{align*}
&C_{1} =\left(\frac{p}{p-1}\right)^p\int_{T}^{(T+a)/2}
\left\vert V_{k}^{\prime}\right\vert^{p}\phi\,dt,
&&C_{2} =\int_{T}^{(T+a)/2}\left\vert V_{k}\right\vert^{p}\etaaT^{p}\phi\,dt.
\end{align*}
Note that $C_1,C_2$ are independent of $k$.
Then, we have:
\begin{align*}
\dfrac{\int_{0}^{a}\left\vert V_{k}^{\prime}\right\vert^{p}\phi\,dt}
{\int_{0}^{a}V_{k}^{p}{\eta} _{a T}^{p}\phi\,dt}
=&\left(\frac{p-1}{p}\right) ^{p} \frac{\int_{\frac{1}{k}}^{T}
\frac{\phi^{-\frac{1}{p-1}}}{\int_{t}^{a}\phi^{-\frac{1}{p-1}}\,d\sigma}\,dt+C_{1}}
{A_{k}+\int_{\frac{1}{k}}^{T}\frac{\phi^{-\frac{1}{p-1}}}{\int_{t}^{a}\phi^{-\frac{1}{p-1}}
\,d\sigma}\,dt+C_{2}}\\
=&\left(\frac{p-1}{p}\right)^{p}\frac{\int_{\frac{1}{k}}^{T}
\frac{\phi^{-\frac{1}{p-1}}}{\int_{t}^{a}\phi^{-\frac{1}{p-1}}\,d\sigma}\,dt+C_{1}}
{\frac{1}{p-1}+o(1)+\int_{\frac{1}{k}}^{T}\frac{\phi^{-\frac{1}{p-1}}}
{\int_{t}^{a}\phi ^{-\frac{1}{p-1}}\,d\sigma}\,dt+C_{2}},
\end{align*}
where $A_k$ is defined in \eqref{A_k}.
This establishes the claim.
\end{proof}
\section{Proofs of Theorem~\ref{thm:sphere} and Theorem~\ref{thm:Fubini}}
\label{sec:wholespace}
In this section we apply Proposition~\ref{One_dim_Hardy_Tr} in order to
prove Theorem~\ref{thm:sphere} and Theorem~\ref{thm:Fubini}. In what follows
we assume that $1<p<n$. We let $a\in(0,\pi )$ and we take $\phi =\widetilde{
\phi }$, where
\begin{equation*}
\widetilde{\phi }(t)=\sin ^{n-1}(t).
\end{equation*}
We note that $\widetilde{\phi}$ satisfies assumptions \eqref{phibasic} with
$\delta =n-p$. The weight function corresponding to $\widetilde{\phi }$
defined according to \eqref{eta} is given by \eqref{tildeetaa}, namely
\begin{equation*}
\widetilde{\eta }_{a}(t)=\frac{\left( \sin t\right) ^{-\frac{n-1}{p-1}}}{
\int_{t}^{a}\left( \sin \sigma \right) ^{^{-\frac{n-1}{p-1}}}d\sigma }.
\end{equation*}
Furthermore, $\widetilde{\phi}$ is twice differentiable and we have
\begin{equation*}
\left( \log \widetilde{\phi }\right) ^{\prime \prime }(t)
=-\frac{n-1}{\sin^{2}t}
\end{equation*}
for all $t\in (0,\pi)$. In particular, $\widetilde{\phi }$ satisfies
assumption (\ref{log_concave}). 
Using L'Hospital's rule we have: 
\begin{equation}
\frac{p-1}{n-p}\lim_{t\rightarrow 0^{+}}t\widetilde{\eta}_{aT}(t)
=\frac{p-1}{n-p}\lim_{t\rightarrow 0^{+}}
\frac{t\left(\sin t \right)^{-\frac{n-1}{p-1}}}{\int_{t}^{a}
\left( \sin \sigma \right) ^{-\frac{n-1}{p-1}}\,d\sigma }
=\frac{p-1}{n-p}\lim_{t\rightarrow 0^{+}}\frac{\left(1+\circ (1)\right)
t^{-\frac{n-p}{p-1}}}{\int_{t}^{a}\sin\sigma^{-\frac{n-1}{p-1}}\,d\sigma }=1,
\label{rho_asym}
\end{equation}
and therefore \eqref{rhoasympt} follows.
The following elementary facts will be used in the sequel.
Recall that for $x=(x_1,\ldots,x_n,x_{n+1})\in\R^{n+1}$
we set $x_{n+1}=|x|\cos\theta$.
\begin{lemma}
\label{grad_radial_funct}
Let $\Omega\subset\mathbb{S}^{n}$ and suppose that $u:\Omega \rightarrow \mathbb{R}$
depends on $\theta$ only. Then:
\begin{align}
\label{mainformulae}
&|\nabla u|^{2}=\left(\frac{\partial u}{\partial \theta }\right)^{2},
&&\int_{\mathcal{B}(a )}u(\theta )\,dV=\omega _{n-1}\int_0^a
u(\theta )\widetilde\phi(\theta)\,d\theta
\end{align}
where $\omega _{n-1}=(2\pi )^{n/2}/\Gamma (n/2)$ denotes the volume of $\mathbb{S}^{n-1}$.
\end{lemma}
We shall also need the following basic facts concerning spherical
rearrangements, see, e.g., \cite{BaeT,S}.
For every $a \in \lbrack 0,\pi]$, let
\begin{equation}
\label{Atheta}
A(a)=|\mathcal{B}(a )|
=\omega_{n-1}\int_{0}^{a }\widetilde{\phi}(\theta)\,d\theta.
\end{equation}
Let $\Omega\subset\mathbb{S}^{n}$ be an open set and let $u:\Omega\rightarrow \mathbb{R}$
be a measurable function. For every $t>0$, let
\begin{equation*}
\mu(t)=|\left\{x\in \Omega :\ \left\vert u(x)\right\vert >t\right\}|
\end{equation*}
denote the distribution function of $u$. Then the decreasing rearrangement $u^{\ast}$
of $u$ is defined by
\begin{equation}
u^{\ast}(s)=\inf\left\{t\geq0:\>\mu (t)\leq s\right\}
\label{u^*}
\end{equation}
for every $s\in \left[0,|\Omega|\right]$.
Let $\Omega^{\star}=\mathcal{B}(a^{\star})$, where $a^{\star}=A^{-1}(|\Omega |)$.
Then, the spherical rearrangement $u^{\star}$ of $u$ is defined by
\begin{equation*}
u^{\star}(x)=u^{\ast}(A(\theta)),\qquad x\in\Omega^{\star}.
\end{equation*}
It follows that $u^{\star}$ is a decreasing function of $\theta$, and that
its level sets are geodesic balls  (spherical caps) centered at
$N=(0,0,\ldots,1)\in\mathbb S^n$.
Since $|u|$ and $u^{\star}$ have the same distribution function, we have
\begin{equation*}
\int_{\Omega}\left\vert u\right\vert^{q}\,dV=\int_{\Omega^{\star}}
\left(u^{\star }\right)^{q}\,dV,
\end{equation*}
for all $q\geq 1$.
We shall use two standard inequalities involving rearrangements.
The following lemma is a special case of the well-known
Hardy-Littlewood inequality and may be found, e.g., in \cite{BS}, Theorem~2.2 p.~44.
\begin{lemma}[Hardy-Littlewood inequality]
\label{lem:HL}
Let $\Omega\subset\mathbb{S}^{n}$ be an open set and suppose that
$u,v:\Omega\rightarrow \mathbb{R}$ are measurable and finite a.e. Then,
\begin{equation}
\int_{\Omega }uv\,dV\leq\int_{\Omega^{\star}}u^{\star}v^{\star}\,dV.
\label{HL}
\end{equation}
\end{lemma}
The following inequality is a special case of the P\'{o}lya-Szeg\H{o} principle,
and may be found in \cite{Au}, Proposition~2.17, p.~41, see also
\cite{S}, Theorem p.~325.
\begin{lemma}[P\'{o}lya-Szeg\H{o} principle]
\label{lem:PSz}
Let $q\geq 1$ and let $u\in W^{1,q}(\mathbb S^n)$. Then,
\begin{equation}
\int_{\mathbb{S}^{n}}\left\vert \nabla u\right\vert ^{q}\,\dvol\geq
\int_{\mathbb{S}^{n}}\left\vert \nabla u^{\star }\right\vert^{q}\,\dvol.
\label{PS}
\end{equation}
\end{lemma}
We can now prove Theorem~\ref{thm:sphere}.
\begin{proof}[Proof of Theorem~\ref{thm:sphere}]
For every $u\in W_0^{1,p}(\Omega)$, in view of Lemma~\ref{grad_radial_funct}
and Lemma~\ref{lem:PSz}, we have:
\begin{align*}
\int_{\Omega}\left\vert\nabla u\right\vert^{p}\,\dvol
\ge\int_{\Omega^{\star }}\left\vert \nabla u^{\star}
\right\vert ^{p}\,\dvol
=\cn\int_{0}^{a^{\star }}\left\vert
\frac{\partial u^{\star }}{\partial\theta}\right\vert^{p}\widetilde{\phi }(\theta)d\theta.
\end{align*}
On the other hand, in view of Lemma~\ref{lem:HL}, we have:
\begin{align*}
\int_{\Omega }\left\vert
u\right\vert^{p}\rho_{a^{\star}}^p\,\dvol
\le
\int_{\Omega^{\star }}\left\vert u^{\star}\right\vert^{p}\rho_{a^{\star}}^p\,\dvol
=\cn\int_{0}^{a^{\star}}\left\vert
u^{\star}\right\vert^{p}\rho_{a^{\star}}^{p}\widetilde{\phi}(\theta)\,d\theta.
\end{align*}
Therefore, it suffices to show that
\begin{align*}
\int_{0}^{a^\star}\left\vert \frac{\partial u^{\star }}{\partial\theta}\right\vert^{p}\widetilde{\phi}
(\theta)\,d\theta-\left(\frac{n-p}{p}\right)^{p}\int_{0}^{a^\star}
\left\vert u^{\star}\right\vert^{p}\rho_{a^\star}^{p}
\widetilde\phi(\theta)\,d\theta
\geq 0.
\end{align*}
The above inequality holds by
definition of $\rho_{a^{\star}}$, as in \eqref{rho*}, and by
Proposition~\ref{One_dim_Hardy_Tr}.
\par
In order to show that the constant $\left(\dfrac{n-p}{p}\right)^{p}$
is sharp it suffices to use, as test
functions, the sequence
$\{\widetilde{V}_{k}\}_{k \in \mathbb{N}}$
obtained by setting $\phi =\widetilde{\phi}$,
$a=a^\star$ and $T=\widetilde{T}$
in \eqref{U_tilde_k}.
Namely,
\begin{equation}
\widetilde{V}_{k}(\theta)=
\begin{cases}
\left( \int_{\frac{1}{k}}^{a^\star}
\widetilde{\phi} ^{-\frac{1}{p-1}}d\sigma \right)^{\frac{p-1}{p}}
&\mathrm{if\ }\theta\in \left[ 0,\frac{1}{k}\right)
\\
\left(\int_{t}^{a^\star}\widetilde{\phi}^{-\frac{1}{p-1}}\,d\sigma\right)^{\frac{p-1}{p}}
&\mathrm{if\ }\theta\in\left[\frac{1}{k},\widetilde{T}\right)\\
\left( \int_{\widetilde{T}}^{a^\star}\widetilde{\phi}^{-\frac{1}{p-1}}\,d\sigma \right)^{\frac{p-1}{p}}
\frac{2t-{a^\star}-\widetilde{T}}{\widetilde{T}-{a^\star}}
&\mathrm{if\ }\theta\in\left[\widetilde{T},(a^\star+\widetilde{T})/2\right)\\
0
&\mathrm{if\ }\theta\in\left[(a^\star+\widetilde{T})/2,a^\star\right]
\end{cases}.
\label{V_hat_k}
\end{equation}
Now the proof of Theorem~\ref{thm:sphere} is complete.
\end{proof}
In order to prove Theorem~\ref{thm:Fubini} we use a Steiner-type symmetrization
on $\R_+^{n+1}$ with respect to the angular variables.
See, e.g., \cite{ADLT,K} for the main results on Steiner symmetrization.
Let $u\in C_0^1(\R_+^{n+1})$.
For every fixed $r>0$
we consider the function obtained by restricting $u$
to $\mathbb{S}^n\cap\R_+^{n+1}$.
Namely, we consider the function
\begin{equation}
\label{restr}
\Theta\in\mathbb{S}^n\cap\R_+^{n+1}\rightarrow u(r,\Theta),
\end{equation}
where $\Theta=(\theta_1,...,\theta_{n-1}, \theta)$ is the set of all angular variables.
We denote by  $u^*(r,\cdot)$ the decreasing rearrangement of the function
in \eqref{restr}, according to the definition given in \eqref{u^*}.
Finally  we introduce the Steiner rearrangement $u^\sharp$ of $u$ as follows:
\begin{equation} \label{u^sharp}
u^\sharp(r,\theta)=u^*(r,A(\theta)), \quad \theta \in [0,\pi/2],
\end{equation}
where $A(\theta)$ is defined in \eqref{Atheta}.
We denote by $g_{r}$ the standard metric on $\mathbb{S}_{r}^{n}$
and by $dV_r$ the volume element on $\mathbb{S}_{r}^{n}$.
Then, we have $D_\Theta u=\nabla_{g_r}u$
and therefore, in view of Lemma~\ref{grad_radial_funct}
and a rescaling argument,
\begin{equation}
\label{|Du^Star|}
|D_\Theta u^\sharp(r,\theta)|^p
=\frac{1}{r^p}\left|\frac{\partial u^\sharp}{\partial\theta}\right|^p.
\end{equation}
We claim
that:
\begin{equation}
\frac{\zeta ^{p}(x)}{|x|^{p}}\in L^{1}(\mathbb{R}_{+}^{n+1}\cap B_{R})
\label{zetaL1}
\end{equation}
for every $R>0$, where $\zeta$ is the weight function appearing in the
statement of Theorem~\ref{thm:Fubini}. Indeed, writing $x=(x^{\prime },x_{n+1})$, in view of
\eqref{rho_asym}, we have for some $C>0$:
\begin{equation*}
\frac{\zeta (x)}{|x|}\leq \frac{C}{\theta r}\leq \frac{C}{r\sin\theta}=\frac{C}{|x^{\prime }|}.
\end{equation*}
Consequently, for any $R>0$ we have:
\begin{equation*}
\int_{|x^{\prime }|,|x_{n+1}|<R}\frac{\zeta ^{p}(x)}{|x|^{p}}
\,dx=\int_{0}^{R}\,dx_{n+1}\int_{|x^{\prime }|<R}\frac{\zeta ^{p}(x)}{|x|^{p}}
\,dx^{\prime }.
\end{equation*}
Now \eqref{zetaL1} follows in view of the assumption $p<n$.
\begin{proof}[Proof of Theorem~\ref{thm:Fubini}]
By density, it suffices to consider $u\in C_{0}^{1}\left( \mathbb{R}_{+}^{n+1}\right)$.
By rescaling, if $\Omega\subset\mathbb{S}_{r}^{n}$ and $u:\Omega \rightarrow \mathbb{R}$
depends on $\theta$ only, then rescaling \eqref{mainformulae} we obtain
\begin{align}
\label{mainrformulae}
&|\nabla _{g_{r}}u|^{2}=\frac{1}{r^2}\left(\frac{\partial u}{\partial\theta}\right)^{2}
&&\int_{\mathcal{B}_{r}(\alpha)}u(\theta )\,dV_{r}=\omega_{n-1}r^{n}
\int_{0}^\alpha u(\theta )\widetilde{\phi}(\theta)\,d\theta,
\end{align}
By Fubini's Theorem and in view of Lemma~\ref{lem:PSz},
we have:
\begin{align}
\int_{\mathbb{R}_{+}^{n+1}}\left\vert D_{\Theta}u\right\vert ^{p}\,dx
\label{R_+_1}
=\int_0^{+\infty}\int_{\mathbb S_r^n}|\nabla_{g_r}u|^p\,d\sigma_r
\ge\int_0^{+\infty}\int_{\mathbb S_r^n}|\nabla_{g_r}u^\sharp|^p\,d\sigma_r
\end{align}
where $u^{\sharp }=u^{\sharp}(r,\theta)$ is defined in \eqref{u^sharp}.
Consequently, from \eqref{|Du^Star|}, \eqref{R_+_1} and
in view of Theorem \ref{thm:sphere}
with $a^{\star }=\pi/2$, we derive
\begin{align*}
\int_{\mathbb{R}_{+}^{n+1}}\left\vert D_{\Theta }u\right\vert ^{p}\,dx
\geq &\cn\int_{0}^{+\infty }\left( \int_{0}^{\frac{\pi }{2}}
\frac{1}{r^{p}}\left( \frac{\partial u^{\star }}{\partial \theta}
\right)^{p}\sin ^{n-1}\theta\,d\theta\right) r^{n}\,dr \\
\geq &\cn\left(\frac{n-p}{p}\right) ^{p}\int_{0}^{+\infty
}\left( \frac{1}{r^{p}}\int_{0}^{\frac{\pi }{2}}\left\vert u^{\star
}\right\vert ^{p}\zeta ^{p}\sin ^{n-1}\theta\,d\theta\right)
r^{n}\,dr \\
\geq &\left( \frac{n-p}{p}\right) ^{p}\int_{\mathbb{R}_{+}^{n+1}}\frac{
\left\vert u\right\vert ^{p}}{r^{p}}\zeta ^{p}\,dx.
\end{align*}
We are left to prove sharpness. To this end, we consider the sequence
$u_{k}(\theta,r)=\Theta_{k}(\theta)R_{k}(r),\text{ }k\in \mathbb{N}$,
where $R_{k}\in C_{0}(0,+\infty )$ satisfies
$R_k>0$ and $R_{k}^{p}(r)\overset{\ast}{\rightharpoonup}\delta_{1}(r)$,
weakly in the sense of measures. Here $\delta_1$ denotes the Dirac mass
on $(0,+\infty)$ centered at $r=1$, and
$\Theta_{k}(\theta)=\widetilde{V}_{k}(\theta)$, where $\widetilde V_k$ is
the sequence defined in \eqref{V_hat_k}, with
$a^{\star}= \pi /2$.
We have
\begin{equation*}
\lim_{k\rightarrow +\infty }\int_{0}^{+\infty
}R_{k}^{p}(r)r^{n}\,dr=\lim_{k\rightarrow +\infty }\int_{0}^{+\infty
}R_{k}^{p}(r)r^{n-p}\,dr=1.
\end{equation*}
Now, the claim follows since
\begin{align*}
\frac{\int_{\mathbb{R}_{+}^{n+1}}\left\vert D_{\Theta }u_{k}\right\vert^{p}\,dx}
{\int_{\mathbb{R}_{+}^{n+1}}\left\vert u_{k}\right\vert ^{p}\dfrac{\zeta ^{p}}{r^{p}}\,dx}
=&\frac{\int_{0}^{+\infty }R_{k}^{p}(r)r^{n}\,dr}{\int_{0}^{+\infty }R_{k}^{p}(r)r^{n-p}\,dr}
\frac{\int_{\mathbb{S}^{n}_{r}\cap\mathbb{R}_{+}^{n+1}}
\left\vert\Theta_{k}^{\prime }(\theta)\right\vert^{p}\,dv_{g_{r}}}
{\int_{\mathbb{S}^{n}_{r}\cap\mathbb{R}_{+}^{n+1}}\left\vert\Theta_{k}(\theta)\right\vert^{p}\zeta^{p}\,dv_{g_{r}}}\\
=&\frac{\int_{0}^{\pi/2}\left\vert\Theta_{k}^{\prime}(\theta)\right\vert^{p}\widetilde{\phi}\,d\theta}
{\int_{0}^{\pi/2}\left\vert \Theta _{k}(\theta)\right\vert ^{p}\zeta ^{p}\widetilde{\phi}\,d\theta}+o(1)
=\left(\frac{n-p}{p}\right)^{p}+o(1),
\end{align*}
where $o(1)$ vanishes as $k\to\infty$.
\end{proof}

\end{document}